\documentclass[10pt]{amsart}
\usepackage{amsfonts,amssymb,amscd,amsmath,enumerate,verbatim,calc}
\usepackage{mathrsfs}
\numberwithin{equation}{section}
\newtheorem{theorem}{Theorem}[section]
\newtheorem{corollary}[theorem]{Corollary}
\newtheorem{lemma}[theorem]{Lemma}
\newtheorem{proposition}[theorem]{Proposition}
\newtheorem{definition}[theorem]{Definition}

\newtheorem{example}[theorem]{Example}

\DeclareMathOperator{\E}{E}

\begin{document}

\title[ Para-Kahler hom-Lie algebroids ]
 { Para-Kahler hom-Lie algebroids  }

\bibliographystyle{amsplain}

 \author[E.  Peyghan, L. Nourmohammadifar and A.  Makhlouf]{ Esmaeil Peyghan, Leila Nourmohammadifar and Abdenacer Makhlouf}
\address{E.P. and L.N. : Department of Mathematics, Faculty of Science, Arak University,
Arak, 38156-8-8349, Iran.\\
A.M. Universit\'e de Haute Alsace, IRIMAS-d\'epartement de math\'ematiques, Mulhouse, France}
\email{e-peyghan@araku.ac.ir,\ l.nourmohammadifar@gmail.com, abdenacer.makhlouf@uha.fr}

\keywords{ Hom-algebroid, left symmetric hom-algebroid, Hom-Lie algebroid, para-K\"{a}hler hom-Lie algebroid}

\subjclass[2010]{17A30, 17D25, 53C15, 53D05.}


\begin{abstract}
 The purpose of this paper is to  study hom-algebroids, among them  left symmetric hom-algebroids and  symplectic hom-algebroids by providing some  characterizations and geometric interpretations. Therefore, we introduce and study   para-K\"{a}hler hom-Lie algebroids and show  various  properties and examples
including  these structures.
\end{abstract}

\maketitle





\section*{Introduction}
Hom-type algebras was motivated by $\sigma$-deformations of some algebras of vector fields like Witt and Virasoro algebras. The first instances appeared in papers by physicists, where it was noticed that the obtained structures satisfy modified  Jacobi condition. The main feature of Hom-type algebras is that  usual identities are twisted by a structure map (a homomorphism). 

Hom-Lie algebroids was first studied by Laurent-Gengoux and Teles in \cite{LT}, they mainly showed a one-to-one correspondence with hom-Gerstenhaber algebras. Then  Cai, Liu and Sheng, in \cite{CLS}, changed slightly the definition and introduced various related structures. They showed that there is a natural  Hom-Lie algebroid structure on the pullback bundle of a Lie algebroid with respect to a diffeomorphism $\varphi:M\rightarrow M$. Moreover they introduced the notion of Hom-Poisson tensor on $\mathcal{C}^\infty(M)$ and showed that there is a Hom-Lie algebroid structure on  $\varphi^!T^*M$ associated to Hom-Poisson structure, providing an interesting geometric interpretation. They also consider dual structures and discuss Hom-Lie bialgebroids. See also \cite{LSBC2,LSBC3} for more recent results.
In \cite{mandal}, Mandal and Kumar Mishra studied adjoint functors between the category of hom-Gerstenhaber algebras and the category of hom-Lie-Rinehart algebras, some geometric applications  and cohomology.

The aim of this paper is to make some geometric developments of these structures. We recall in Section 1  basic definitions.  In Section 2, we introduce hom-algebroids and  left symmetric hom-algebroids. Also, we deal with  symplectic hom-algebroids and show that there exists a natural hom-connection, called left symmetric connections,  which induces a left symmetric hom-algebroid structure.  Moreover we describe  hom-Levi-Cevita connections for which we provide some properties. The last section includes the main results of this paper. First, we introduce and provide examples of  almost product, para-complex and para-Hermitian structures on hom-Lie algebroids.  Then consider  para-K\"{a}hler hom-Lie algebroids, discuss their properties and relationships with various structures.
\section{Preliminaries}
A hom-algebra $(V,\cdot , \phi)$    consists of a linear space $V$, a bilinear map 
$\cdot : V \times V \rightarrow V$ and an algebra morphism $\phi : V \rightarrow V$.
Also a hom-Lie algebra is a triple $(\mathfrak{g}, [\cdot  , \cdot ],{ \phi‎‎_\mathfrak{g}})$ consisting of a linear space $\mathfrak{g}$, a
 bilinear map $[\cdot , \cdot ]: \mathfrak{g}\times\mathfrak{g}\rightarrow\mathfrak{g}$ and an algebra morphism ${ \phi‎‎_\mathfrak{g}}:\mathfrak{g}\rightarrow \mathfrak{g}$ such that
\begin{align*}
[u,v]=-[v,u],\ \ \ \ \ \ 
\circlearrowleft_{u,v,w}[\phi_{\mathfrak{g}}(u),[v,w]]=0,
\end{align*}
for any $u,v,w\in\mathfrak{g}$. This algebra is called regular if ${ \phi_\mathfrak{g}}$ is non-degenerate.

Let $A\rightarrow M$ be
a vector bundle of rank $n$. Denote by $\Gamma(A)$ the $C^\infty(M)$-module of sections of $A\rightarrow M$. A \textit{hom-bundle} is a triple $(A\rightarrow M, \varphi,\phi_A)$ consisting of a vector bundle $A\rightarrow M$, a smooth map $\varphi: M \rightarrow M$ and an algebra morphism $\phi_A:\Gamma(A)\rightarrow  \Gamma(A)$ satisfying
\[
\phi_A(fX)=\varphi^*(f)\phi_A(X),
\]
for any $X\in \Gamma(A)$ and $f\in C^{\infty}(M)$ (in this case, $\phi_A$ is called a linear $\varphi^*$-function). If $\varphi$ is a diffeomorphism and $\phi_A$ is an invertible map, then the hom-bundle $(A\rightarrow M, \varphi,\phi_A)$ is called invertible. Considering
 $\varphi^!TM$ as a bull back bundle of $\varphi$ over $M$ and $Ad_{\varphi^*}:\Gamma(\varphi^!TM)\rightarrow \Gamma(\varphi^!TM)$  given by
\[
Ad_{\varphi^*}(X)=\varphi^*\circ X\circ(\varphi^*)^{-1},
\]
for any $X\in\Gamma(\varphi^!TM)$ \cite{CLS}, then the  triple $(\Gamma(\varphi^!TM),\varphi,Ad_{\varphi^*})$ is an example of hom-bundles. Note that  $\Gamma(\varphi^!TM)$ can be identified with  $ Der_{\varphi^*,\varphi^*}({C^\infty(M))}$, i.e.
\[
X(fg)=X(f)\varphi^*(g)+\varphi^*(f)X(g),\ \ \ \ \ \ \ \ \ \forall X\in\Gamma(\varphi^!TM), \forall f,g\in C^\infty(M).
\]
A bundle map $\rho:A\rightarrow B$ between two hom-bundles $(A\rightarrow M,\varphi,\phi_A)$ and $(B\rightarrow M,\varphi,\phi_B)$  is called a hom-bundles morphism if the following condition holds
\begin{align*}
\rho\circ\phi_A=&\phi_B\circ\rho.
\end{align*}
\begin{definition}\cite{CLS}
	A \textit{hom-Lie algebroid } is a tuple $(A,\varphi,\phi_A, [\cdot,\cdot]_A, a_A)$
	such that
	$(A\rightarrow M, \varphi,\phi_A)$ is a hom-bundle,  $(\Gamma(A), [\cdot,\cdot]_A,\phi_A)$ is  a hom-Lie algebra on the section space $\Gamma(A)$, $a_A:A\rightarrow \varphi^!TM$ is a bundle map called the anchor map and if moreover, we have 
	\[
	[X, fY]_A=\varphi^*(f)[X, Y]_A+a_A(\phi_A(X))(f)\phi_A(Y),\ \ \ \forall X, Y\in\Gamma(A),\ \forall f\in C^\infty(M),
	\]
	when 
	$a_A:\Gamma(A)\rightarrow\Gamma(\varphi^!TM)$ is
	the representation of hom-Lie algebra $(\Gamma(A), [\cdot,\cdot]_A,\phi_A)$ on $C^{\infty}(M)$ with respect to $\varphi^*$ induced by the anchor map.
\end{definition}
A subspace $B\subset A$ is a hom-Lie subalgebroid of $(A,\varphi,\phi_A, [\cdot,\cdot]_A, a_A)$ if $\phi_A(B)\subset B$ and 
\[
[X,Y]_A\in \Gamma(B), \\\ \forall X,Y\in \Gamma(B).
\]
In the sequel, we always assume that the hom-bundle $(A\rightarrow M, \varphi,\phi_A)$ is invertible.
The operator $d^{A}:\Gamma(\wedge^q A^*)\rightarrow\Gamma(\wedge^{q+1} A^*)$ defined by 
\begin{align*}
& d^{A}f(z) =a_A (z)f, \\
& d^A\omega(z_1,\ldots, z_{q+1})=\sum_{i=1}^{q+1}(-1)^{i+1}a_A(z_i)\omega( \phi_A^{-1}(z_1),\ldots, \widehat{\phi_A^{-1}(z_i)}, \ldots, \phi_A^{-1}(z_{q+1}))
\notag \\
& \hspace{2cm}�+\sum_{i<j}(-1)^{i+j}\phi_A^\dagger(\omega)([\phi_A^{-1}(z_i), \phi_A^{-1}(z_j)]_A, z_1, \ldots,  \hat{z_i}, \ldots,\hat{ z_j}, \ldots,  z_{q+1}),  \label{290}
\end{align*}%
is called \emph{\ the exterior
	differentiation ope\-ra\-tor for the exterior differential algebra }of the hom-Lie algebroid $(A,{\varphi},\phi_A,[\cdot,\cdot]_A, a_A)$.
Let $p,q>0$. The operator 
\begin{equation*}
\begin{array}{ccc}
\Gamma(\wedge ^{p}{A})\times \Gamma(\wedge ^{q}{A}) & ^{%
	\underrightarrow{~\ \ [\cdot,\cdot]_A~\ \ }} & \Gamma(\wedge ^{p+q-1}{A})%
\end{array}%
,
\end{equation*}%
given by 
\begin{align}  \label{C2}
\lbrack u_{1}\wedge \ldots \wedge u_{p},v_{1}\wedge \ldots \wedge
v_{q}]_A
=&(-1)^{p+1}\sum_{i=1}^{p}%
\sum_{j=1}^{q}(-1)^{i+j}[u_{i},v_{j}]_{A}\wedge \phi_A(u_{1})\wedge \ldots \wedge 
\notag \\
&\wedge \widehat{\phi_A(u_{i})}\wedge \ldots
\wedge \phi_A(u_{p}) \wedge  \phi_A(v_{1})\wedge \ldots \wedge \widehat{ \phi_A(v_{j})}\wedge \ldots
\wedge \phi_A(v_{q}),
\end{align}%
is  called  hom-Schouten bracket.
\label{X4} The  hom-Schouten bracket satisfies the following
conditions 
\begin{align*}
1. ~&[u,v]_A=-(-1)^{(p-1)(q-1)}[v,u]_A, \\
2. ~&[u,v\wedge z]_A=[u,v]_A\wedge  \phi_A(z)+(-1)^{(p-1)q} \phi_A(v)\wedge \lbrack u,z]_A, 
\end{align*}%
where $u\in\Gamma(\wedge ^{p}{A}), v\in\Gamma(\wedge ^{q}{%
	A})$ and $z\in\Gamma(\wedge ^{r}{A})$.

For any $u\in\Gamma({A})$, the  operator 
\begin{equation*}
\begin{array}{ccc}
L_u:\Gamma(\wedge ^{p}{A})\rightarrow\Gamma(\wedge ^{p}{A}),
\end{array}%
\end{equation*}
given by $L_u(v)=[u, v]_A$ is called the Lie derivative.

Let $z\in L( A)$. The operator $L_{z}:\Gamma(\wedge A^*)\rightarrow \Gamma(\wedge A^*)$
defined by 
\begin{align}
&L_{z}(f) =a_A (\phi_A(z))(f),   \\
&L_{z}\omega (z_{1},\cdots ,z_{q}) =a_A (\phi_A(z))\omega (\phi_A^{-1}(z_{1}),\cdots ,\phi_A^{-1}(z_{q}))-%
\overset{q}{\underset{i=1}{\sum }}\phi_A^\dagger(\omega) (z_{1},\cdots
,[z,\phi_A^{-1}(z_{i})]_{A},\cdots ,z_{q}),  \label{form2}
\end{align}%
for any $f\in C^\infty({M})$, $\omega \in \Gamma(\wedge^q A^*)$ and $%
z_{1},\cdots ,z_{q}\in \Gamma(A),$ is called the \textit{covariant Lie
	derivative with respect to the element $z$}.

Let $(A,\varphi,\phi_A, [\cdot,\cdot]_A, a_A)$ be a finite-dimensional hom-Lie algebroid and $< ,>$
be a bilinear symmetric non-degenerate form  on $A$.  We say that $A$ admits a pseudo-Riemannian metric $<,>$ if  the following equation is satisfied
\begin{equation}\label{AM302}
<\phi_A(X), \phi_A(Y)>=\varphi^*<X,Y>, \ \ \ \forall X,Y\in \Gamma(A).
\end{equation}
In this case $(A,\varphi,\phi_A, [\cdot,\cdot]_A, a_A,<,>)$ is called pseudo-Riemannian hom-Lie algebroid.

Assume that $(A\rightarrow M,\varphi,\phi_A)$ is a hom-bundle. An
$A$-connection on a hom-bundle $(E\rightarrow M,\varphi,\phi_E)$ is an operator 
\[
\nabla:\Gamma(A)\times\Gamma(E)\rightarrow \Gamma(E),
\] 
satisfying:\\
i) $\nabla_{X+Y}Z=\nabla_XZ+\nabla_YZ$,\\
ii) $\nabla_{fX}Z=\varphi^*(f)\nabla_XZ$,\\
iii) $\nabla_X(Z'+Z)=\nabla_XZ'+\nabla_XZ$,\\
iv) $\nabla_X(fZ)=\varphi^*(f)\nabla_XZ+a_A(\phi_A(X))(f)\phi_E(Z)$,\\
for any $f\in C^{\infty}(M ), X,Y\in\Gamma(A), Z',Z\in\Gamma(E) $.\\
Moreover, if $\phi_A$ is an isomorphism, then there exists a unique connection $\nabla$ on it such that:
\begin{align}
[X,Y]=&\nabla_XY-\nabla_YX,\label{P5}\\
a_A(\phi_A(X))<Y,Z>=&<\nabla_XY,\phi_A(Z)>+<\phi_A(Y),\nabla_XZ>\label{L2},
\end{align}
for any $X,Y,Z\in \Gamma(A)$.
This connection is called the hom-Levi-Civita connection \cite{PNA}, which is given by Koszul's formula
\begin{align}\label{Koszul}
2<\nabla_XY,\phi_A(Z)>=&a_A(\phi_A(X))<Y,Z>+a_A(\phi_A(Y))<Z,X>-a_A(\phi_A(Z))<X,Y>\\
&+<[X,Y],\phi_A(Z)>+<[Z,X],\phi_A(Y)>+<[Z,Y],\phi_A(X)>\nonumber.
\end{align}
\section{ hom-algebroids and left symmetric structures of  hom-algebroids}
In this section,  we introduce hom-algebroids and  left symmetric hom-algebroids.  Also, we deal with symplectic   hom-Lie algebroids and show that one may associate   natural connections defined by the symplectic structure and  that lead to left symmetric hom-algebroids.  They are called  hom-left symmetric connections.
\begin{definition}
	A hom-algebroid structure on a hom-bundle $(A\rightarrow M ,\varphi,\phi_A)$ is a pair 
	consisting of a hom-algebra structure $(\Gamma(A),\cdot_A,\phi_A)$ on the section space $\Gamma(A)$ and a bundle morphism
	$a_A:A\rightarrow \varphi^!TM$, called the anchor, such that the following conditions are satisfied
	\begin{equation}
	\left\{
	\begin{array}{cc}
	X\cdot_A(fY)=\varphi^*(f)(X\cdot_A Y)+a_A(\phi_A(X))(f)\phi_A(Y),\\ 
	\hspace{-3.59cm}(fX)\cdot_A Y=\varphi^*(f)(X\cdot_A Y),\\
\hspace{-3.5cm}{\varphi^*\circ a_A(X)}
={a_A(\phi_A(X))\circ\varphi^*},
	\end{array}
	\right.
	\end{equation}
	for any $ X,Y\in \Gamma(A)$ and $f\in C^\infty(M)$. We denote a hom-algebroid by $(A,\varphi,\phi_A, \cdot_A,a_A)$.
\end{definition}
\begin{definition}\label{MA}
A hom-algebroid $(A,\varphi,\phi_A, \cdot_A,a_A)$ is called left symmetric if it
 endowed with a bilinear skew-symmetric nondegenerate form  $\Omega\in \Gamma(\wedge^2A^*)$ such that for any 
 $X,Y,Z\in \Gamma(A)$,
\begin{align*}
\Omega(ass_{\phi_A}(X,Y,Z)-ass_{\phi_A}(Y,X,Z),\phi_A^2(Z'))=a_A(\phi_A^2(Z))a_A(\phi_A(Z'))\Omega(X,Y)
-\varphi^*a_A(Z\cdot_AZ')\Omega(X,Y),
\end{align*}
where
\[
ass_{\phi_A}(X,Y,Z)=(X\cdot_A Y)\cdot_A\phi_A(Z)-\phi_A(X)\cdot_A(Y\cdot_AZ).
\]
 In this case, the product $\cdot_A$ is called a hom-left symmetric product on the hom-bundle $(A\rightarrow M ,\varphi,\phi_A)$.
\end{definition}
\begin{example}
 A left symmetric hom-algebroid over a one-point set with the
zero anchor, is a hom-left-symmetric algebra \cite{PN2}.
\end{example}
\begin{definition}
A symplectic   hom-Lie algebroid is a  hom-Lie algebroid $(A,\varphi,\phi_A, [\cdot,\cdot]_A, a_A)$ endowed with a bilinear skew-symmetric
non-degenerate form $\omega$ which is a $2$-hom-cocycle, i.e.
\begin{align}\label{1E}
\phi_A^\dagger‎ \omega=\omega,\ \ \ d^A\omega=0.
\end{align}
In this case, $\omega$ is called symplectic structure on $A$ and $(A,\omega)$ is called  symplectic  hom-Lie algebroid.
\end{definition}
The conditions (\ref{1E}) are equivalent to 
\begin{align}\label{AM100}
&a_A(\phi_A(X))\omega(Y,Z)-a_A(\phi_A(Y))\omega(X,Z)+a_A(\phi_A(Z))\omega(X,Y)\\
&-\omega([X,Y]_A,  \phi_A(Z)) + \omega([X,Z]_A,\phi_A(Y)) -\omega([Y,Z]_A, \phi_A(X))=0\nonumber, \ \ \ \forall X,Y,Z\in\Gamma(A).
\end{align}
\begin{definition}
A representation of a hom-Lie algebroid 
$(A,\varphi,\phi_A, [\cdot,\cdot]_A, a_A)$ on a  hom-bundle $(E, \varphi, \phi_E)$ (shortly $E$) with respect to a linear $\varphi^*$-function  map $\mu:\Gamma(E)\rightarrow\Gamma(E)$ is a hom-bundle map $\rho: A\rightarrow \mathfrak{D}(E)$ such that for all
$X,Y\in \Gamma(A)$, the following
equalities are satisfied
\begin{equation}\label{PP1}
\left\{
\begin{array}{cc}
&\hspace{-5.3cm}a_{_{\mathfrak{D}(E)}}\circ \rho=a_A\circ\phi_A,\\
&\hspace{-4.5cm}\rho(\phi_{A}(X))\circ \mu=\mu\circ \rho(X),\\
&\hspace{-.2cm}\rho([X,Y]_{A})\circ \mu=\rho(\phi_{A}(X))\circ \rho(Y)-\rho(\phi_{A}(Y))\circ \rho(X).
\end{array}
\right.
\end{equation}
We denote a representation of $A$ by $(E; \mu, \rho)$ \cite{PN}.
\end{definition}
Let $(E^*,\varphi,\phi_E ^\dagger)$ be  the dual hom-bundle of the hom-bundle $(E,\varphi,\phi_E)$, then the dual map of $\rho$ is the map 
$\widetilde{\rho}:A\rightarrow \mathfrak{D}(E^*)$ given by 
\begin{equation}\label{Mg}
\prec\widetilde{ \rho}(X)(\xi),Y\succ =a_A(\phi_A(X))(\prec\xi,\mu^{-1}(Y)\succ)-\varphi^*\prec\xi,\rho(\phi_A^{-1}(X))(\mu^{-2}(Y))\succ, 
\end{equation}
for any $X\in\Gamma(A), Y\in\Gamma(E)$ and $
\xi\in\Gamma(E^*)$. 
Moreover, it is easy to see that $\widetilde{\rho}$   is a representation of $(A,\varphi,\phi_A, [\cdot,\cdot]_A, a_A)$ on hom-bundle $(E^*,\varphi,\phi_E ^\dagger)$ with respect to $\mu^\dagger$.
\begin{corollary}\label{P1}
If  $L=\widetilde{\rho}$, then  $(A;\phi_A,L)$ and $(A^*;\phi^\dagger_A,\widetilde{L})$ are representations of $A$, where $L$ and $\widetilde{L}$ are Lie derivative  and covariant Lie
derivative of $A$, respectively.
\end{corollary}

\begin{theorem}\label{P4}
Let  $(A,\omega)$ be a symplectic  hom-Lie algebroid. Then there exists a
$A$-connection
 $\nabla^{\bold{a}}:\Gamma(A)\times \Gamma(A)\rightarrow \Gamma(A)$ given by
\begin{align}\label{P2}
\omega(\nabla^{\bold{a}}_XY,\phi_A(Z)) =a_A(\phi_A(X))\omega(Y,Z)-\omega(\phi_A(Y),[X,Z]_A),
\end{align}
and satisfying
\begin{align*}
&i)\ \ \omega(\nabla^{\bold{a}}_XY -\nabla^{\bold{a}}_YX-[X,Y]_A,\phi_A(Z)) =-a_A(\phi_A(Z))\omega(X,Y)\label{1},\\
&ii)\ \ \omega(\phi_A^2(Z'),[\nabla^{\bold{a}}_XY -\nabla^{\bold{a}}_YX-[X,Y]_A,\phi_A(Z)]_A)=-a_A(\phi_A^2(Z))a_A(\phi_A(Z'))\omega(X,Y)\\
&\ \ \ \ \ +a_A({ \phi‎‎_A}(\nabla^{\bold{a}}_ZZ'))\omega(\phi_A(X),\phi_A(Y)),\nonumber\\
&iii)\ \ \nabla^{\bold{a}}_{{ \phi‎‎_A}(Y)}\nabla^{\bold{a}}_XZ-\nabla^{\bold{a}}_{{ \phi‎‎_A}(X)}\nabla^{\bold{a}}_YZ+\nabla^{\bold{a}}_{\nabla^{\bold{a}}_XY}{ \phi‎‎_A}(Z)-\nabla^{\bold{a}}_{\nabla^{\bold{a}}_YX}{ \phi‎‎_A}(Z)\\
&\ \ \ \ \ \ =[\nabla^{\bold{a}}_XY -\nabla^{\bold{a}}_YX-[X,Y]_A,\phi_A(Z)]_A\nonumber,
\end{align*}
for any $X,Y,Z\in \Gamma(A)$.
\end{theorem}
\begin{proof}
Let $\flat : ‎‎\Gamma(A) \rightarrow ‎‎\Gamma(A^*)$ be an isomorphism
given by $\flat(X)=\omega(X,\cdot)$ for any $X\in ‎‎\Gamma(A)$. Then 
 $\prec \flat(X), Y‎\succ=\omega(X,Y)$, for any $Y\in ‎‎‎‎\Gamma(A)$.
As $\flat$ is  bijective and using Corollary \ref{P1}, we can define
$
\nabla^{\bold{a}}_XY= \flat^{-1} \widetilde{L}_{ X}\flat(Y).
$
Thus, we get
\begin{align*}
\omega(\nabla^{\bold{a}}_XY,{ \phi_A}(Z))=& \prec \widetilde{L}_X \flat(Y),{ \phi_A}(Z)‎\succ=a_A(\phi_A(X))(\langle\flat(Y),Z\rangle)-\langle\phi_A^\dagger(\flat(Y)),L_XZ\rangle\\
=&a_A(\phi_A(X))\omega(Y,Z)-\phi_A^\dagger(\omega)(\phi_A(Y),[X,Z]_A)\\
=&a_A(\phi_A(X))\omega(Y,Z)-\omega(\phi_A(Y),[X,Z]_A),
\end{align*}
which gives (\ref{P2}).
Using (\ref{AM100}) and (\ref{P2}), we obtain
 \begin{align*}
&\omega(\nabla^{\bold{a}}_XY -\nabla^{\bold{a}}_YX,\phi_A(Z)) =a_A(\phi_A(X))\omega(Y,Z)-a_A(\phi_A(Y))\omega(X,Z)\\
&+ \omega([X,Z]_A,\phi_A(Y)) -\omega([Y,Z]_A, \phi_A(X))=-a_A(\phi_A(Z))\omega(X,Y)+\omega([X,Y]_A,  \phi_A(Z)) \nonumber,
\end{align*}
that gives us (i). Similarly we get (ii). We have
\begin{align}
0=\omega({ \phi‎‎^2_A}(Z'),[{ \phi‎‎_A}(X),[Y,Z]_A]_A+[{ \phi‎‎_A}(Y),[Z,X]_A]_A+[{ \phi‎‎_A}(Z),[X,Y]_A]_A).
\end{align}
Using (\ref{P2}), we get
\begin{align}\label{PE3}
\omega({ \phi‎‎^2_A}(Z'),&[{ \phi‎‎_A}(X),[Y,Z]_A]_A=-\omega(\nabla^{\bold{a}}_{{ \phi‎‎_A}(X)}{ \phi‎‎_{A}}(Z'),{ \phi‎‎_A}[Y,Z]_A]_A)+a_A(\phi^2_A(X))\omega(\phi_A(Z'),[Y,Z]_A)\\
=&\omega(\nabla^{\bold{a}}_{{ \phi‎‎_A}(Y)}\nabla^{\bold{a}}_XZ',{ \phi‎‎_A^2}(Z))-a_A(\phi^2_A(Y))\omega(\nabla^{\bold{a}}_XZ',\phi_A(Z))+a_A(\phi^2_A(X))\omega(\phi_A(Z'),[Y,Z]_A)\nonumber.
\end{align}
Similarly
\begin{align}\label{PE4}
\omega({ \phi‎‎^2_A}(Z'),[{ \phi‎‎_A}(Y),[Z,X]_A]_A
=&-\omega(\nabla^{\bold{a}}_{{ \phi‎‎_A}(X)}\nabla^{\bold{a}}_YZ',{ \phi‎‎_A^2}(Z))+a_A(\phi^2_A(X))\omega(\nabla^{\bold{a}}_YZ',\phi_A(Z))\\
&-a_A(\phi^2_A(Y))\omega(\phi_A(Z'),[X,Z]_A)\nonumber.
\end{align}
Applying (\ref{P2}) and (ii), we obtain
\begin{align*}
\omega({ \phi‎‎^2_A}(Z'),&[{ \phi‎‎_A}(Z),[X,Y]_A]_A)=\omega({ \phi_A^2‎‎}(Z'),[{ \phi‎‎_A}(Z),\nabla^{\bold{a}}_XY-\nabla^{\bold{a}}_YX]_A)-a_A(\phi^2_A(Z))a_A(\phi_A(Z'))\omega(X,Y)\\
&+a_A(\phi_A(\nabla^{\bold{a}}_ZZ'))\omega(\phi_A(X),\phi_A(Y))\\
=&\omega(\nabla^{\bold{a}}_{\nabla^{\bold{a}}_XY}{ \phi‎‎_A}(Z'),{ \phi‎‎_A^2}(Z))-\omega(\nabla^{\bold{a}}_{\nabla^{\bold{a}}_YX}{ \phi‎‎_A}(Z'),{ \phi‎‎_A^2}(Z))-a_A(\phi^2_A(Z))a_A(\phi_A(Z'))\omega(X,Y)\\
&+a_A(\phi_A(\nabla^{\bold{a}}_ZZ'))\omega(\phi_A(X),\phi_A(Y))-a_A(\phi_A(\nabla^{\bold{a}}_XY))\omega(\phi_A(Z'),\phi_A(Z))\\
&+a_A(\phi_A(\nabla^{\bold{a}}_YX))\omega(\phi_A(Z'),\phi_A(Z)).
\end{align*}
But, on the other hand 
\begin{align*}
&-a_A(\phi_A(\nabla^{\bold{a}}_XY))\omega(\phi_A(Z'),\phi_A(Z))+a_A(\phi_A(\nabla^{\bold{a}}_YX))\omega(\phi_A(Z'),\phi_A(Z))\\
&=-a_A(\phi_A([X,Y]_A))\omega(\phi_A(Z'),\phi_A(Z))-a_A(\phi_A(\nabla^{\bold{a}}_{Z}Z'))\omega(\phi_A(X),\phi_A(Y))\nonumber\\
&+a_A(\phi_A(\nabla^{\bold{a}}_{Z'}Z))\omega(\phi_A(X),\phi_A(Y))-a_A(\phi_A([{Z'},Z]_A))\omega(\phi_A(X),\phi_A(Y))\nonumber.
\end{align*}
The two above equations  imply 
\begin{align}\label{PE5}
&\omega({ \phi‎‎^2_A}(Z'),[{ \phi‎‎_A}(Z),[X,Y]_A]_A)
=\omega(\nabla^{\bold{a}}_{\nabla^{\bold{a}}_XY}{ \phi‎‎_A}(Z'),{ \phi‎‎_A^2}(Z))-\omega(\nabla^{\bold{a}}_{\nabla^{\bold{a}}_YX}{ \phi‎‎_A}(Z'),{ \phi‎‎_A^2}(Z))\\
&-a_A(\phi^2_A(Z'))a_A(\phi_A(Z))\omega(X,Y)+a_A(\phi_A(\nabla^{\bold{a}}_{Z'}Z))\omega(\phi_A(X),\phi_A(Y))\nonumber.
\end{align}
Setting (\ref{PE3}),  (\ref{PE4}) and (\ref{PE5}) in (2.6) and using (ii), we conclude (iii).
\end{proof}
\begin{corollary}
 The connection $\nabla^\bold{a}$ in the  Theorem \ref{P4} induces a left symmetric structure on the symplectic  hom-Lie algebroid $(A,\omega)$.
\end{corollary}
\begin{proof}
	From (iii) of Theorem \ref{P4}, we have 
	\begin{align*}
	&\omega(\nabla^{\bold{a}}_{{ \phi‎‎_A}(Y)}\nabla^{\bold{a}}_XZ-\nabla^{\bold{a}}_{{ \phi‎‎_A}(X)}\nabla^{\bold{a}}_YZ+\nabla^{\bold{a}}_{\nabla^{\bold{a}}_XY}{ \phi‎‎_A}(Z)-\nabla^{\bold{a}}_{\nabla^{\bold{a}}_YX}{ \phi‎‎_A}(Z),\phi_A^2(Z'))\\
	&\ \ \ \ \ \ =\omega([\nabla^{\bold{a}}_XY -\nabla^{\bold{a}}_YX-[X,Y]_A,\phi_A(Z)]_A,\phi_A^2(Z'))\nonumber.
	\end{align*}
	According to (ii) of Theorem \ref{P4}, the above equation implies
	\begin{align*}
&\omega(\nabla^{\bold{a}}_{{ \phi‎‎_A}(Y)}\nabla^{\bold{a}}_XZ-\nabla^{\bold{a}}_{{ \phi‎‎_A}(X)}\nabla^{\bold{a}}_YZ+\nabla^{\bold{a}}_{\nabla^{\bold{a}}_XY}{ \phi‎‎_A}(Z)-\nabla^{\bold{a}}_{\nabla^{\bold{a}}_YX}{ \phi‎‎_A}(Z),\phi_A^2(Z'))\\
&\ \ \ \ \ =a_A(\phi_A^2(Z))a_A(\phi_A(Z'))\omega(X,Y)-a_A({ \phi‎‎_A}(\nabla^{\bold{a}}_ZZ'))\omega(\phi_A(X),\phi_A(Y)).
	\end{align*}
Hence,	by Definition \ref{MA}, we conclude that $\nabla^{\bold{a}}$ is a hom-left symmetric product.
	\end{proof}
 We consider $\nabla^\bold{a}$  as the hom-left symmetric connection associated with $(A,\omega)$.
\section{ Para-K\"{a}hler hom-Lie algebroids}
In this section, first we introduce almost product, para-complex and para-Hermitian structures on hom-Lie algebroids.  Then consider  para-K\"{a}hler hom-Lie algebroids and discuss their properties.
\begin{definition}
    An almost product structure on a hom-Lie algebroid  $(A,\varphi,\phi_A, [\cdot,\cdot]_A, a_A)$, is an isomorphism $ K : ‎‎A \rightarrow A$ such that
    \[
    (\phi_A\circ K)^2= Id_{‎‎A},\ \ \ \  \phi_A\circ K=K\circ \phi_A.
    \]
We denote an almost product hom-Lie algebroid by $(A, K)$.
\end{definition}
  So, one can write $A$ as $A=A^1\oplus A^{-1}$, such that
    \[
    A^1:=ker({{ \phi‎‎_A}\circ K-Id_A}),\ \ A^{-1}:=ker({ \phi_A}\circ K+Id_{A} ).
    \]
   Also $K$ is called an \textit{almost para-complex} structure on $A$, if $A^1$ and $A^{-1}$ have the same dimension $n$  (in this case the dimensional of $A$ is even). We define the \textit{Nijenhuis torsion} of ${{ \phi_A}\circ K}$ as follows
\begin{equation}\label{w1}
    N_{{ \phi_A}\circ K}(X,Y)=[({{ \phi_A}\circ K})X, ({{ \phi_A}\circ K})Y]_A-{{ \phi_A}\circ K}[({{ \phi_A}\circ K})X, Y]_A-{{ \phi_A}\circ K}[X,({{ \phi_A}\circ K})Y]_A+[X,Y]_A,
\end{equation}
for all $X,Y\in \Gamma(A)$. In the sequel for simplicity, we set $N=N_{{ \phi_A}\circ K}$.  If $N=0$, then the
 almost product (almost para-complex) structure is called \textit{product (para-complex)}. 
\begin{example}
Let $M$ be
a  $2n$-dimensional manifold endowed with an $n$-codimensional foliation $F$. Considering  hom-bundle $(\varphi^!TM,\varphi,Ad_{\varphi^*})$, there
is the decomposition $\varphi^!TM=\varphi^!TF\oplus\varphi^!T^\perp F$ where $(\varphi^!TF,\varphi,Ad_{\varphi^*})$ is the hom-bundle (the
bull back bundle of $\varphi$ over  the leaves of F) and  $(\varphi^!T^\perp F,\varphi,Ad_{\varphi^*})$ is the transversal hom-bundle of $F$. We define the bracket on $\varphi^!TM$ as follows
\begin{align*}
&[u,v]_{{\varphi^*}}=\varphi^*\circ u\circ (\varphi^*)^{-1}\circ v\circ(\varphi^*)^{-1}-\varphi^*\circ v\circ(\varphi^*)^{-1}\circ u\circ (\varphi^*)^{-1},\ \ \ \\
&[\bar{u},\bar{v}]_{{\varphi^*}}=\varphi^*\circ \bar{u}\circ (\varphi^*)^{-1}\circ \bar{v}\circ(\varphi^*)^{-1}-\varphi^*\circ \bar{v}\circ(\varphi^*)^{-1}\circ \bar{u}\circ (\varphi^*)^{-1},\ \ \ \ \ \\
&[{u},\bar{v}]_{{\varphi^*}}=0,
\end{align*} 
for any $u,v\in \Gamma(\varphi^!TF),\bar{u},\bar{v}\in \Gamma(\varphi^!T^\perp F)$. Since
\begin{align*}
&Ad_{{\varphi^*}}[u+\bar{u},v+\bar{v}]_{{\varphi^*}}=[Ad_{{\varphi^*}}(u+\bar{u}),Ad_{{\varphi^*}}(v+\bar{v})]_{{\varphi^*}},\\
&[Ad_{{\varphi^*}}(u+\bar{u}),[v+\bar{v},z+\bar{z}]_{{\varphi^*}}]_{{\varphi^*}}+c.p.=0.
\end{align*}
Thus $(\varphi^!TF\oplus\varphi^!T^\perp F, [\cdot,\cdot]_{{\varphi^*}},Ad_{\varphi^*})$ is a hom-Lie algebra. Also we see that  $(\varphi^!TF\oplus\varphi^!T^\perp F,\varphi,Ad_{{\varphi^*}}\oplus Ad_{{\varphi^*}}, [\cdot,\cdot]_{{\varphi^*}},Id)$ is a hom-Lie algebroid.
If the isomorphism $K$ is given  as follows
\[
K(u)=Ad_{{\varphi^*}}^{-1}(u), \ \ \ \ \ K(\bar{u})=-Ad_{{\varphi^*}}^{-1}(\bar{u}),
\]
then using the above equations, we have
\begin{align*}
&(K\circ Ad_{{\varphi^*}})(u)=u=(Ad_{{\varphi^*}}\circ K)(u),\ \ \ (K\circ Ad_{{\varphi^*}})(\bar{u})=-\bar{u}=(Ad_{{\varphi^*}}\circ K)(\bar{u}).
\end{align*}
Also, $(K\circ Ad_{{\varphi^*}})^2=Id$. Therefore K is an almost product structure on $\varphi^!TM$. Since $\varphi^!TF$ and $\varphi^!T^\perp F$ have the same dimension $n$, we also
deduce that K is an almost para-complex structure on $\varphi^!TM$. Moreover,  K is a para-complex structure on $\varphi^!TM$ because
\[
N(u,v)=N(\bar{u},\bar{v})=N(u,\bar{v})=0.
\]
\end{example}
\begin{definition}\label{L100}
    An almost para-hermitian structure on a  hom-Lie algebroid  $(A,\varphi,\phi_A, [\cdot,\cdot]_A, a_A)$ is a pair $(K, <,>)$ consisting of an almost para-complex structure
    and a pseudo-Riemannian metric $<,>$, such that
    \begin{equation}\label{EL}
    <({ \phi_A}\circ K)X, ({ \phi_A}\circ K)Y>=- <X,Y>, \ \ \ \forall X,Y \in \Gamma(A).
    \end{equation}
Moreover, If $N=0$, then the pair $(K, <,>)$ is called para-hermitian structure. In this case, $(A, K, <,>)$ is called para-Hermitian hom-Lie algebroid.
\end{definition}
\begin{definition}
A para-K\"{a}hler hom-Lie algebroid is an  almost para-hermitian hom-Lie algebroid $(A,\varphi,\phi_A, [\cdot,\cdot]_A, a_A,K,<,>)$ such that ${ \phi_ A}\circ K$ is invariant with respect to the hom-Levi-Civita connection $\nabla$, i.e., 
\begin{align}\label{AM002}
\nabla_X { \phi_A}( KY)={ \phi_A}(K(\nabla_XY)),\ \ \ \forall X,Y\in \Gamma(A).
\end{align}
\end{definition}
The above condition  implies
\begin{align}\label{AM2}
\nabla_{{ \phi_A}( KX)}( { \phi_A}( KY))={ \phi_A}(K(\nabla_{{ \phi_A}( KX)}Y)),\ \ \ 
\nabla_XY={ \phi_A}(K(\nabla_X\phi_A(KY))).
\end{align}
\begin{example}\label{04}
Let $(M,\varphi,\pi)$ be a   hom-Poisson manifold \cite{CLS}. We consider $E:=\varphi^!TM\oplus\varphi^!T^*M$, where 
 $(\varphi^!TM,\varphi,Ad_{\varphi^*})$ and  $(\varphi^!T^* M,\varphi,Ad^\dagger_{\varphi^*})$ are hom-bundles. We define the bracket and linear map $\Phi$ on $\E$ as follows
\begin{align*}
&[X+\alpha,Y+\beta]=([X,Y]_{{\varphi^*}}+\lbrack \alpha ,\beta ]_{_{{\Pi^\sharp}}}),\\
&\Phi(X+\alpha)=Ad_{\varphi^*}(X)+Ad^\dagger_{\varphi^*}(\alpha), 
\end{align*}
where
\begin{align*}
&[X,Y]_{{\varphi^*}}=\varphi^*\circ X\circ (\varphi^*)^{-1}\circ Y\circ(\varphi^*)^{-1}-\varphi^*\circ Y\circ(\varphi^*)^{-1}\circ X\circ (\varphi^*)^{-1},\ \ \ \ \\
&\lbrack \alpha ,\beta ]_{_{{\Pi^\sharp}}}=L_{\Pi^\sharp \left( \alpha \right)
}\beta -L_{\Pi ^\sharp\left( \beta \right) }\omega-d^{A}(\Pi \left( \alpha,\beta \right)),\ \ \ \ \ \ 
\end{align*} 
for any $ X,Y\in \Gamma(\varphi^!TM)$ and $\alpha,\beta\in\Gamma(\varphi^!T^*M)$. Easily  we see that 
\begin{align*}
&\Phi[X+\alpha,Y+\beta]=[\Phi(X+\alpha),\Phi(Y+\beta],\\
&[\Phi(X+\alpha),[Y+\beta,z+\gamma]]+c.p.=0.
\end{align*}
Thus $(E, [\cdot,\cdot],\Phi)$ is a hom-Lie algebra. Let $a_E:E\rightarrow\varphi^!TM $ be the bundle map defined by
\[
a_E(X+\alpha)=X+\pi^\sharp(\alpha).
\]
Then $(E, \varphi,\Phi, [\cdot,\cdot],a_E)$ is a hom-Lie algebroid.
 If the metric $<,>$ is determined as
\begin{align*}
&<X+\alpha,Y+\beta>=\alpha(Y)+\beta(X), \ \ \\
&(a_E\circ\Phi)(Z)(\alpha(Y))=(a_E\circ\Phi)(\gamma)(\alpha(Y))=0,
\end{align*}
then we have 
\[
<\Phi(X+\alpha),\Phi(Y+\beta)>=\varphi^*(\alpha(Y))+\varphi^*(\beta(X)).
\]
Assume the isomorphism $K:E\rightarrow E$  given by
\[
K(X)=Ad_{{\varphi^*}}^{-1}(X), \ \ \ \ \ K(\alpha)=-(Ad^\dagger_{{\varphi^*}})^{-1}(\alpha).
\]
So using the above equations we deduce
\begin{align*}
&(K\circ Ad_{{\varphi^*}})(X)=X=(Ad_{{\varphi^*}}\circ K)(X),\ \ \ (K\circ Ad^\dagger_{{\varphi^*}})(\alpha)=-\alpha=(Ad^\dagger_{{\varphi^*}}\circ K)(\alpha),\\
&(K\circ Ad_{{\varphi^*}})^2(X)=X,\ \ \ \ (K\circ Ad^\dagger_{{\varphi^*}})^2(\alpha)=\alpha.
\end{align*} 
Therefore $K$ is an almost product structure on $E$. As $\varphi^!TM$ and $\varphi^!T^*M$ have the same dimension $n$, thus  K is an almost para-complex structure on $E$. Also, we have
\[
N(X,Y)=N(\alpha,\beta)=N(X,\alpha)=0,
\]
that is  $K$ is a para-complex structure on E. It is easy to check that 
\begin{align*}
<(K\circ Ad_{{\varphi^*}})(X),(K\circ Ad_{{\varphi^*}})(Y)>=&0=<X,Y>,\\
<(K\circ Ad^\dagger_{{\varphi^*}})(\alpha),(K\circ Ad^\dagger_{{\varphi^*}})(\beta)>=&0=<\alpha,\beta>,\\
<(K\circ Ad_{{\varphi^*}})({X}),(K\circ Ad^\dagger_{{\varphi^*}})(\alpha)>=&-\alpha(X)=-<{X},\alpha>,
\end{align*}
and so  $(E, \varphi,\Phi, [\cdot,\cdot],Id+\pi^\sharp, <,>, K)$ is a para-Hermitian hom-Lie algebroid. From Koszul's formula given by (\ref{Koszul}), we get
\[
\nabla_XY=\frac{1}{2}[X,Y]_{\varphi^*},\ \ \ \nabla_{\alpha}\beta=\frac{1}{2}[\alpha,\beta]_{\pi^\sharp},\ \ \ \nabla_{\alpha}{X}=\nabla_{{X}}\alpha=0.
\]
It is easy to see that the hom-Levi-Civita connection computed  above satisfies  (\ref{AM002}). Thus \\ $(E, \varphi,\Phi, [\cdot,\cdot],Id+\pi^\sharp, <,>, K)$ is a para-K\"{a}hler hom-Lie algebroid.
\end{example}
\begin{proposition}\label{AR}
Let  $(A,\varphi,\phi_A, [\cdot,\cdot]_A, a_A,K,<,>)$ be a para-K\"{a}hler hom-Lie algebroid. If we consider
\begin{equation}\label{*}
\Omega(X,Y)=<({ \phi_A}\circ K)X,Y>, \ \ \ \forall X,Y\in \Gamma(A),
\end{equation}
then   $(A, \Omega)$ is a symplectic hom-Lie algebroid.
\end{proposition}
\begin{proof}It turns out that 
(\ref{AM302}) and (\ref{*}) imply $\phi_A^\dagger\Omega=\Omega$, because
\begin{align*}
\Omega({ \phi‎‎_A}(X),{ \phi‎‎_A}(Y))=&<({ \phi‎‎_A}\circ K){ \phi‎‎_A}(X), { \phi‎‎_A}( Y)>=\varphi^*<(K\circ{ \phi‎‎_A})X,Y>=\varphi^*<({ \phi‎‎_A}\circ K)X,Y>\\
=&\varphi^*\Omega(X,Y).
\end{align*}
Applying  (\ref{AM100}) and (\ref{*}), we obtain
\begin{align}\label{AM300}
&\Omega([X,Y]_A,{ \phi‎‎_A}(Z))+\Omega([Y,Z]_A,{ \phi‎‎_A}(X))+\Omega([Z,X]_A,{ \phi‎‎_A}(Y))-a_A(\phi_A(X))\omega(Y,Z)\\
&\ \ \ \ +a_A(\phi_A(Y))\omega(X,Z)-a_A(\phi_A(Z))\omega(X,Y)\nonumber\\
&=-<[X,Y]_A,({ \phi‎‎_A}\circ K) ({ \phi‎‎_A} (Z))>-<[Y,Z]_A,({ \phi‎‎_A}\circ K) ({ \phi‎‎_A} (X))>\nonumber\\
&\ \ \ \ -<[Z,X]_A,({ \phi‎‎_A}\circ K) ({ \phi‎‎_A} (Y))>-a_A(\phi_A(X))\omega(Y,Z)
+a_A(\phi_A(Y))\omega(X,Z)-a_A(\phi_A(Z))\omega(X,Y)\nonumber\\
&=-<\nabla_XY,({ \phi‎‎_A}\circ K) ({ \phi‎‎_A} (Z))>+<\nabla_YX,({ \phi‎‎_A}\circ K) ({ \phi‎‎_A} (Z))>-<\nabla_YZ,({ \phi‎‎_A}\circ K) ({ \phi‎‎_A} (X))>\nonumber\\
&\ \ \ \ +<\nabla_Z Y,({ \phi‎‎_A}\circ K) ({ \phi‎‎_A} (X))>-<\nabla_Z X,({ \phi‎‎_A}\circ K) ({ \phi‎‎_A} (Y))>+<\nabla_XZ,({ \phi‎‎_A}\circ K) ({ \phi‎‎_A}( Y))>\nonumber\\
&\ \ \ \ -a_A(\phi_A(X))\omega(Y,Z)
+a_A(\phi_A(Y))\omega(X,Z)-a_A(\phi_A(Z))\omega(X,Y)\nonumber,
\end{align}
for any $X,Y,Z\in\Gamma(A)$. Using (\ref{L2}) and (\ref{AM2}) we conclude
\begin{align*}\label{AM301}
&<\nabla_XY, ({ \phi‎‎_A}\circ K)({ \phi‎‎_A} (Z))>=<({ \phi‎‎_A}\circ K)(\nabla_X ({ \phi‎‎_A}\circ K)(Y)), ({ \phi‎‎_A}\circ K)({ \phi‎‎_A} (Z))>\\
&=-<\nabla_X({ \phi‎‎_A}\circ K)Y, ({ \phi‎‎_A} Z)>=-a_A(\phi_A(X))\omega(Y,Z)+<\nabla_XZ,{ \phi‎‎_A} ({ \phi‎‎_A}\circ K)(Y)>.\nonumber
\end{align*}
Setting the above equation in (\ref{AM300}), we get the assertion i.e., $d^A\Omega=0$.
\end{proof}
According to  the above proposition, a para-K\"{a}hler hom-Lie algebroid has two connections, the hom-Levi-Civita connection and the hom-left
    symmetric connection $\nabla^\bold{ a}$ associated with $(\Gamma(A), [\cdot, \cdot]_A, \phi_{A}, \Omega)$.

\begin{proposition}\label{304}
Let $(A=A^1\oplus A^{-1},\varphi,\phi_{A}, [\cdot, \cdot]_A, <,>, K)$ be a para-K\"{a}hler hom-Lie algebroid. Then

i) \ Nijenhuis torsion $N$ is zero,

ii)\  $A^1$ and $A^{-1}$ are subalgebroids isotropic with respect to $< ,>$, and Lagrangian with
respect to $\Omega$,

iii)\   $\nabla_X\Gamma(A^1)\subset \Gamma(A^1) $ and $\nabla_X\Gamma(A^{-1})\subset\Gamma(A^{-1}) $, for any $X\in\Gamma(A)$
($\nabla$ is the hom-Levi-Civita connection),

iv)\  for any $X\in\Gamma(A^1)$, $\phi‎‎_A(X)\in\Gamma(A^1)$ and for any $\bar{X}\in\Gamma(A^{-1})$, 
$\phi_A(\bar{X})\in\Gamma(A^{-1})$  i.e.,
\[
\phi_{A}(X+\bar{X})=\phi_{A^1}(X)+\phi_{A^{-1}}(\bar{X}),
\]
\ \ \ \ v) \ $(A, K <,>)$ is a para-Hermitian  hom-Lie algebroid,

vi)\  $A^{-1}\simeq(A^1)^*$ ($(A^1)^*$ is dual space of $A^1$) and endomorphisms $\phi_{A^{-1}}$ and $(\phi_{A^{1}})^\dagger$ are the same.
\end{proposition}
\begin{proof}
Let $X,Y\in\Gamma( A)$. Then using  (\ref{AM2}), we obtain
\begin{align*}
N(X,Y)=&[({{ \phi‎‎_A}\circ K})X, ({{ \phi_A}\circ K})Y]-({{ \phi_A}\circ K})[({{ \phi_A}\circ K})X, Y]-({{ \phi_A}\circ K})[X,({{ \phi_A}\circ K})Y]+[X,Y]\\
=&
\nabla_{\phi_A( KX)} \phi_ A( KY)-\nabla_{\phi_A( KY)} \phi_A( KX)-\nabla_{\phi_A( KX)} \phi_A( KY)+\nabla_Y X-\nabla_X Y\\
&
+\nabla_{\phi_A( KY)}\phi_A( KX)+\nabla_X Y-\nabla_Y X=0,
\end{align*}
 which  gives (i).
For any $X,Y\in\Gamma( A^1)$, we have  $<({ \phi_A}\circ K)X,({ \phi_A}\circ K)Y=<X,Y>$. On the other hand, since ${ \phi_A}\circ K$ is skew-symmetric with respect to $<,>$, we conclude that   $<X,Y>=0$. Similarly, $<\bar{X},\bar{Y}>=0$, for any $ \bar{X}, \bar{Y}\in\Gamma(A^{-1})$. Hence, $A^1$ and $A^{-1}$ are isotropic with respect to $<,>$. Now we show that $A^1$ is a Lagrangian subspace of $A$ with respect to $\Omega$. Assume $X\in \Gamma(A^1)$. We have
$
\Omega(X,Y)=0,
$
for any $Y\in\Gamma(A^1)$, which means that
 $X\in\Gamma (A^1)^\perp$ and thus $A^1\subseteq (A^1)^\perp$. Let $0\neq X+\bar{X}\in (A^1)^\perp$ such that $X\in \Gamma(A^1)$ and $\bar{X}\in \Gamma(A^{-1})$. Then, we have
\begin{align*}
0=\Omega(X+\bar{X},Y)=\Omega(X,Y)+\Omega(\bar{X},Y),\ \ \ \forall Y\in \Gamma(A^1).
\end{align*}
As $\Omega(X,Y)=0$, it yields $\Omega(\bar{X},Y)=0$. On the other hand, since $\Omega(\bar{X},\bar{Y})=0$, for any $\bar{Y}\in \Gamma(A^{-1})$, we deduce $\Omega(\bar{X},Y+\bar{Y})=0$, hence $\bar{X}=0$. Therefore $(A^1)^\perp\subseteq A^1$ and consequently $(A^1)^\perp=A^1$. Similarly it follows that $A^{-1}$ is a Lagrangian subspace of $A$ with respect to $\Omega$. Therefore we have $(ii)$. Using (\ref{AM2}) we have
\[
({ \phi_A}\circ K)(\nabla_XY)=\nabla_X{ \phi_A}( KY)=\nabla_XY, \ \ \ \forall X\in\Gamma(A), \forall Y\in \Gamma(A^1).
\]
The above equation implies $\nabla_XY\in \Gamma(A^1)$. Similarly,  $\nabla_X\bar{Y}\in\Gamma(A^{-1})$ for any ${X}\in\Gamma(A)$, which gives us (iii).
To prove (iv), since
\[
(K\circ \phi_{A})(\phi_{A}(X))=(\phi_{A}\circ K\circ \phi_{A})(X)=\phi_{A}(X),\ \ \ \forall X\in\Gamma(A^1),
\]
we have  ${\phi_A}(X)\in \Gamma(A^1)$. Similarly,  ${ \phi_A}(\bar{X})\in \Gamma(A^{-1})$ for any $\bar{X}\in \Gamma(A^{-1})$.
To prove (v), using (i) and (ii) we have
 $K$ is a para-complex structure on $(\Gamma(A), [\cdot, \cdot]_A, \phi_A)$.  Since ${ \phi_A}\circ K$ is skew-symmetric with respect to $<,>$, then we have
 the assertion. To prove (vi), let $\bar{X}\in\Gamma(A^{-1})$ and $\bar{X}^*\in(\Gamma(A^1)^*)$ such that $\prec\bar{X}^*,Y\succ=<\bar{X},Y>$, for all $Y\in\Gamma(A^1)$. Then we have map $A^{-1}\rightarrow(A^1)^*$, $\bar{X}\rightarrow\bar{X}^*$, which is an isomorphism. If we consider $\bar{X}\in \Gamma(A^{-1})$ and $\bar{X}^*$ be the corresponding element of it in $\Gamma(({A^1})^*)$, then for any $Y\in \Gamma(A^1)$ we obtain
\[
\prec(\phi_{A^1})^\dagger(\bar{X}^*), Y\succ=\varphi^*\prec\bar{X}^*, \phi^{-1}_{A^1}(Y)\succ=\varphi^*<\bar{X}, \phi^{-1}_{A^1}(Y)>=<\phi_{A^{-1}}(\bar{X}), Y>=\prec(\phi_{A^{-1}}(\bar{X}))^*, Y\succ,
\]
which gives  $(\phi_{A^1})^\dagger(\bar{X}^*)=(\phi_{A^{-1}}(\bar{X}))^*$, where $(\phi_{A^{-1}}(\bar{X}))^*$ is the corresponding element of $\phi_{A^{-1}}(\bar{X})\in\Gamma( A^{-1})$ in $({A^1})^*$.
\end{proof}
\begin{lemma}
Let $(A,\varphi,\phi_A, [\cdot,\cdot]_A, a_A,K,<,>)$ be a para-K\"{a}hler hom-Lie algebroid. Then
\begin{align*}
\nabla_XY=&\nabla^\bold{a}_XY, \ \ \ \ \forall X,Y\in \Gamma(A^1),\\
 \nabla_{\bar{X}}\bar{Y}=&\nabla^\bold{a}_{\bar{X}}\bar{Y}, \ \ \ \ \forall \bar{X},\bar{Y}\in\Gamma(A^{-1}),
\end{align*}
where $\nabla$ is the  hom-Levi-Civita connection and $\nabla^\bold{a}$ is the hom-left-symmetric
connection.
\end{lemma}
\begin{proof}
 Since $A$ is a symplectic hom-Lie algebroid, then
 using (\ref{P2}) and (\ref{P5}), we get
\begin{align}\label{AM13}
0=&\Omega(\nabla^\bold{a}_XY,{ \phi‎‎_{A}}(Z))-a_A(\phi_A(X))\omega(Y,Z)+\Omega({ \phi‎‎_{A}}(Y),[X,Z])\\
=&\Omega(\nabla^\bold{a}_XY,{ \phi‎‎_A}(Z))-a_A(\phi_A(X))\omega(Y,Z)+\Omega({ \phi‎‎_{A}}(Y),\nabla_XZ-\nabla_ZX)\nonumber,
\end{align}
for any $Z\in \Gamma(A)$ and $X,Y\in\Gamma(A^1)$. Also, (\ref{AM302}),  (\ref{*}) and parts (iii) and (iv) of Proposition \ref{304} imply
\begin{align}\label{AM12}
&\Omega({ \phi‎‎_{A}}(Y),\nabla_XZ)=<({ \phi‎‎_{A}}\circ K)({ \phi‎‎_{A}(Y)}),\nabla_XZ>=<{ \phi‎‎_{A}(Y)},\nabla_XZ>\\
&=a_A(\phi_A(X))\omega(Y,Z)-<{ \phi‎‎_A}(Z), \nabla_XY>=a_A(\phi_A(X))\omega(Y,Z)-\Omega(\nabla_XY,{ \phi‎‎_A}(Z))\nonumber,
\end{align}
and
\[
-\Omega({ \phi_{A}}(Y),\nabla_ZX)=<\nabla_Z X,{ \phi_{A}}(Y)>=\varphi^*<{ \phi_A}(\nabla_Z X),Y>=\varphi^*\Omega(Y,{ \phi_A}(\nabla_Z X))=\Omega({\phi_{A}}(Y), \nabla_Z X).
\]
The above equation implies
\begin{align}\label{AM11}
\Omega({ \phi_{A}}(Y),\nabla_ZX)=0.
\end{align}
Setting (\ref{AM12}) and (\ref{AM11}) in (\ref{AM13}) and using the non-degenerate property of $\Omega$ and $\phi_{A}$,  we have $\nabla_XY=\nabla^\bold{a}_XY$. Similarly, we get the second relation.
\end{proof}
\begin{corollary}
If $(A,\varphi,\phi_A, [\cdot,\cdot]_A, a_A,K,<,>)$ is a para-K\"{a}hler hom-Lie algebroid, then 
$(\Gamma(A^1), \nabla^\bold{a},\phi_{A^1})$ and $(\Gamma(A^{-1}),\nabla^\bold{a},\phi_{A^-{1}})$ are hom-left symmetric algebras. Moreover
they induce hom-Lie algebra structures on $\Gamma(A^1)$ and $\Gamma(A^{-1})$. consequently $(A^1,\varphi,\phi_{A^1}, [\cdot,\cdot]_{A^1},a_{A^1})$ and $(A^{-1},\varphi,\phi_{A^{-1}},[\cdot,\cdot]_{A^{-1}},a_{A^{-1}})$ are  hom-Lie algebroids.
\end{corollary}
In the sequel, since $A^{-1}\simeq({A^1})^*$, we denote the elements of $A^{-1}$  by $\alpha, \beta, \ldots$. Also, to simplify we use ${ \phi_{(A^1)^*}}$ instead of $(\phi_{A^{1}})^\dagger$.
Considering  $(A,\varphi,\phi_A, [\cdot,\cdot]_A, a_A,K,<,>)$ as a para-K\"{a}hler hom-Lie algebroid,  for
  any $X\in\Gamma(A^1), \alpha\in(\Gamma(A^{1})^*)$, we let $\nabla_X$ and $\nabla_\alpha$ be the  hom-Levi-Civita  connection operators on $X$ and $\alpha$, respectively.
\begin{proposition}
With the above notations, we have \\
i)\ $(A^1;\phi_{A^1},\nabla)$ is a representation of the hom-Lie algebroid $A^1$ where
$\nabla:A^1\rightarrow \mathfrak{D}({A}^1)$ with $X\rightarrow \nabla_X$,\\
ii)\  $((A^1)^*;\phi_{(A^1)^*},\nabla)$ is a representation of the hom-Lie algebroid $(A^{1})^*$ where
  $\nabla:(A^1)^*\rightarrow \mathfrak{D}((A^1)^*)$ with $\alpha\rightarrow \nabla_\alpha$.
\end{proposition}
\begin{proof}
Since $A^1$ is  the   isotropic subspace and $\nabla_XY-\nabla_YX=[X,Y]_A$, then using Theorem \ref{P4} we have
\begin{align*}
\nabla^{\bold{a}}_{{ \phi‎‎_{A^1}}(Y)}\nabla^{\bold{a}}_XZ-\nabla^{\bold{a}}_{{ \phi‎‎_{A^1}}(X)}\nabla^{\bold{a}}_YZ+\nabla^{\bold{a}}_{\nabla^{\bold{a}}_XY}{ \phi‎‎_{A^1}}(Z)-\nabla^{\bold{a}}_{\nabla^{\bold{a}}_YX}{ \phi‎‎_A}(Z)=0,\ \ \ \ \forall X,Y,Z\in\Gamma(A^1),
\end{align*}
where $\nabla_XY=\nabla^\bold{a}_XY$. The above equation implies
 \[
\nabla_{[X,Y]}\circ { \phi_{A^1}}=\nabla_{{ \phi_{A^1}}(X)}\circ \nabla_Y-\nabla_{ { \phi_{A^1}}‎(Y)}\circ \nabla_X.
\]
Also, we have ${ \phi_{A^1}}(\nabla_XY)=\nabla_{ \phi_{A^1}(X)} { \phi_{A^1}}(Y)$ and 
\begin{align*}
a_{_{\mathfrak{D}(A^1)}}(\nabla_X)(f)=\nabla_X(f)=a_{A^1}(\phi_{A^1}(X))(f).
\end{align*}
 So (i)  holds. Similarly, we obtain (ii).
\end{proof}
\begin{proposition}\label{qq}
Let $(A,\varphi,\phi_A, [\cdot,\cdot], a_A,K,<,>)$ be a para-K\"{a}hler hom-Lie algebroid. Then $(A^1\oplus(A^1)^*,\varphi,\phi_A, [\cdot,\cdot], a)$ is a hom-Lie algebroid,  where 
\begin{equation}\label{5.2}
\left\{
\begin{array}{cc}
&\hspace{-.1cm}[X+\alpha,Y+\beta]=[X,Y]_{A^1}+\widetilde{\nabla}_X\beta-\widetilde{\nabla}_Y\alpha,\\
&\hspace{-1.5cm}\phi_A(X+\alpha)=\phi_{A^1}(X)+{\phi_{(A^1)^*}}(\alpha),\\
&\hspace{-3.4cm}a(X+\alpha)=a_A(X),
\end{array}
\right.
\end{equation}
for any $X,Y\in \Gamma(A^1), \alpha,\beta\in\Gamma((A^1)^*)$.
\end{proposition}
\begin{proof}
Obviously, we see that
\[
[X+\alpha,Y+\beta]=-[Y+\beta,X+\alpha].
\]
As $((A^1)^* ,{\phi_{(A^1)^*}}, \widetilde{\nabla})$ is a representation of hom-Lie algebroid $A^1$,  we have
\[
\widetilde{\nabla}_{[Y,Z]_{A^1}}\circ {\phi_{(A^1)^*}}=\widetilde{\nabla}_{\phi_{A^1}(Y)}\circ \widetilde{\nabla}_Z-\widetilde{\nabla}_{\phi_{A^1}(Z)}\circ \widetilde{\nabla}_Y.
\]
Using (\ref{5.2}) and the above equation, we obtain
\begin{align*}
&[\phi_A(X+\alpha),[Y+\beta,Z+\gamma]]+c.p.=[\phi_{A^1}(X)+{\phi_{(A^1)^*}}(\alpha),[Y,Z]_{A^1}+\widetilde{\nabla}_Y\gamma-\widetilde{\nabla}_Z\beta]+c.p.\\
&=\big([\phi_{A^1}(X),[Y,Z]_{A^1}]_{A^1}+\widetilde{\nabla}_{\phi_{A^1}(X)}(\widetilde{\nabla}_Y\gamma-\widetilde{\nabla}_Z\beta)-\widetilde{\nabla}_{[Y,Z]_{A^1}}{\phi_{(A^1)^*}}(\alpha)\big)+c.p.\\
&=[\phi_{A^1}(X),[Y,Z]]+c.p.+\widetilde{\nabla}_{\phi_{A^1}(X)}(\widetilde{\nabla}_Y\gamma-\widetilde{\nabla}_Z\beta)+\widetilde{\nabla}_{\phi_{A^1}(Y)}(\widetilde{\nabla}_Z\alpha-\widetilde{\nabla}_X\gamma)+\widetilde{\nabla}_{\phi_{\mathfrak{g}^1}(Z)}(\widetilde{\nabla}_X\beta-\widetilde{\nabla}_Y\alpha)\\
&-\widetilde{\nabla}_{\phi_{A^1}(Y)}\widetilde{\nabla}_Z\alpha+\widetilde{\nabla}_{\phi_{A^1}(Z)} \widetilde{\nabla}_Y\alpha-\widetilde{\nabla}_{\phi_{A^1}(Z)} \widetilde{\nabla}_X\beta+\widetilde{\nabla}_{\phi_{A^1}(X)}\widetilde{\nabla}_Z\beta
-\widetilde{\nabla}_{\phi_{A^1}(X)}\widetilde{\nabla}_Y\gamma+\widetilde{\nabla}_{\phi_{A^1}(Y)}\widetilde{\nabla}_X\gamma=0.
\end{align*}
Therefore $(A^1\oplus(A^1)^*, [\cdot,\cdot] ,\phi_A)$ is a hom-Lie algebra. Since $a_A$ is
the representation of hom-Lie algebroid $(A,\varphi,\phi_A, [\cdot,\cdot]_A,a_A)$, we have
\[
{\varphi^*\circ a(X+\alpha)}={\varphi^*\circ a_A(X)}
={a_A(\phi_A(X))\circ\varphi^*}={\varphi^*\circ a(\phi_A(X+\alpha))},
\]
and
\begin{align*}
a([X+\alpha,Y+\beta])\circ \varphi^*&=a([X,Y]_A+\widetilde{\nabla}_X\beta-\widetilde{\nabla}_Y\alpha)\circ \varphi^*\\
=&a_A([X,Y]_A\circ \varphi^*=a_A(\phi_A(X))\circ a_A(Y)-a_A(\phi_A(Y))\circ a_A(X)\\
=&a(\phi_A(X+\alpha))\circ a(Y+\beta)-a(\phi_A(Y+\beta))\circ a(X+\alpha).
\end{align*}
\end{proof}
\begin{definition}
Let $(A,\varphi,\phi_A, [\cdot,\cdot]_A, a_A)$  be a  hom-Lie algebroid  and $A^*$ be the dual bundle of $A$. A phase space of $A$ is
defined as a  hom-Lie algebroid  $(A\oplus A^*,\varphi,\phi_A\oplus\phi_{A^*},[\cdot,\cdot]_{A\oplus A^*},a)$  consisting of  hom-subalgebroids $A$, $A^*$  and the natural skew-symmetry bilinear form $\omega$
on $A\oplus A^*$ given by
\begin{align}\label{S1}
\omega(X+\alpha,Y+\beta)=\langle\beta,X‎\rangle-\langle\alpha,Y‎\rangle,\ \ \ \ \ \ \ \forall X,Y\in \Gamma(A),\forall \alpha,\beta\in \Gamma(A^*),
\end{align}
which is a symplectic form,  $\phi_{A^*}=\phi_A^\dagger$ and $a(X+\alpha)=a_A(X)$ \cite{PN}.
\end{definition}
\begin{lemma}
If $(A,\varphi,\phi_A, [\cdot,\cdot]_A, a_A,K,<,>)$ is a para-K\"{a}hler hom-Lie algebroid, then the hom-Lie algebroid $A^1\oplus(A^1)^*$ is a phase space of the hom-Lie algebroid $A^1$.
\end{lemma}
\begin{proof}
As the para-K\"{a}hler hom-Lie algebroid $A$ 
 is a symplectic hom-Lie algebroid, then we can write
\begin{align*}
\Omega(X+\alpha,Y+\beta)=<(\phi_{A}\circ K)(X+\alpha),Y+\beta>=<X-\alpha,Y+\beta>=-\prec\alpha,Y‎\succ+\prec\beta,X‎\succ.
\end{align*}
for any $X,Y\in \Gamma(A^1)$ and $\alpha,\beta\in\Gamma((A^1)^*)$. Similar to the proof of  Proposition \ref{AR}, we can prove the following
\begin{align*}
&\Omega(\phi_{A}(X+\alpha),\phi_{A}(Y+\beta))=\varphi^*\Omega(X+\alpha,Y+\beta),
\end{align*}
and
\[
\Omega([X+\alpha,Y+\beta],\phi_A(Z+\gamma))+c.p.=0.
\]
\end{proof}
\begin{example}
We consider the para-K\"{a}hler hom-Lie algebroid  $(E, \varphi,\Phi, [\cdot,\cdot],a_E)$ introduced in Example \ref{04}. $(\varphi^!TM\oplus\varphi^!T^*M, \varphi,\Phi,[\cdot,\cdot], Id)$ is a phase space of the hom-Lie algebroid $\varphi^!TM$.
\end{example}




\end{document}